\documentclass[10pt,a4paper]{amsart}
\usepackage{amsmath,amsthm,amssymb,amsfonts}
\usepackage{tikz-cd}
\usepackage{thmtools,xcolor}
\usepackage[bookmarks=true,hyperindex,pdftex,colorlinks,citecolor=blue,
linkcolor=blue,urlcolor=blue]{hyperref}
\usepackage{mathtools}

\usepackage[english]{babel}
\usepackage{enumitem}
\usepackage{todonotes}
\usepackage{graphicx}

\declaretheoremstyle[
headfont=\color{blue}\normalfont\bfseries,
bodyfont=\color{blue}\normalfont\itshape,
]{colored}

\theoremstyle{plain}
\newtheorem{theorem}{Theorem}[section]
\newtheorem{cor}[theorem]{Corollary}

\newtheorem{prop}[theorem]{Proposition}
\newtheorem{lemma}[theorem]{Lemma}

\theoremstyle{definition}

\newtheorem{example}[theorem]{Example}

\newtheorem{rem}[theorem]{Remark}
\newtheorem{remark}[theorem]{Remark}

\newcommand{\R}{\mathbb{R}}
\newcommand{\N}{\mathbb{N}}
\newcommand{\C}{\mathbb{C}}

\newcommand{\Lin}{\mathcal{L}}

\newcommand{\eps}{\varepsilon}

\DeclareMathOperator{\spann}{span}

\DeclareMathOperator{\sym}{Sym}

\renewcommand{\leq}{\leqslant}
\renewcommand{\geq}{\geqslant}

\title{Group Invariant Separating Polynomials on a Banach space}

\author[J. Falc\'{o}]{Javier Falc\'{o}}
\address[Javier Falc\'{o}]{Departamento de An\'{a}lisis Matem\'{a}tico,
	Universidad de Valencia, Doctor Moliner 50, 46100 Burjasot (Valencia), Spain} \email{francisco.j.falco@uv.es}
\author[D. Garc\'{\i}a]{Domingo Garc\'{\i}a}
\address[Domingo Garc\'{\i}a]{Departamento de An\'{a}lisis Matem\'{a}tico,
	Universidad de Valencia, Doctor Moliner 50, 46100 Burjasot (Valencia), Spain}
\email{domingo.garcia@uv.es}
\author[M. Jung]{Mingu Jung}
\address[Mingu Jung]{Department of Mathematics, POSTECH, Pohang 790-784, Republic of Korea \newline
	\href{http://orcid.org/0000-0003-2240-2855}{ORCID: \texttt{0000-0003-2240-2855} }}
\email{\texttt{jmingoo@postech.ac.kr}}
\author[M. Maestre]{Manuel Maestre}
\address[Manuel Maestre]{Departamento de An\'{a}lisis Matem\'{a}tico,
	Universidad de Valencia, Doctor Moliner 50, 46100 Burjasot
	(Valencia), Spain} \email{manuel.maestre@uv.es}

\thanks{The third author was supported by the Basic Science Research Program through the National Research Foundation of Korea (NRF) funded by the Ministry of Education (NRF-2019R1A2C1003857).  The first, second and fourth authors were supported by MINECO and FEDER Project MTM2017-83262-C2-1-P. The second and fourth authors were also supported by Prometeo PROMETEO/2017/102.}

\date{\today}


\begin{document}
	
	\begin{abstract}
	We study the group invariant continuous polynomials on a Banach space $X$ that separate a given set $K$ in $X$ and a point $z$ outside $K$. We show that if $X$ is a real Banach space, $G$ is a compact group of $\Lin (X)$, $K$ is a $G$-invariant set in $X$, and $z$ is a point outside $K$ that can be separated from $K$ by a continuous polynomial $Q$, then $z$ can also be separated from $K$ by a $G$-invariant continuous polynomial $P$. It turns out that this result does not hold when $X$ is a complex Banach space, so we present some additional conditions to get analogous results for the complex case. We also obtain separation theorems under the assumption that $X$ has a Schauder basis which give  applications to several classical groups. In this case, we obtain characterizations of points which can be separated by a group invariant polynomial from the closed unit ball.  
	\end{abstract}
	
	\maketitle
	
	\section{Introduction}
	The algebras of polynomials invariant under the action of a topological group have been intensively studied during the last years, see for instance \cite{AAGZ,AFGM,CGZ,CGZ2,CGZ3,GGJ,JZ, KVZ, NS} and the references therein. Our aim in this note is to study separation theorems by polynomial that are invariant under a group action. This continues the work started in \cite{AFM} where the authors prove separation theorems that are invariant under the action of a finite group.
	
	A separately continuous action of a topological group $(G, \tau)$ on a topological space $X$ is a group action of $G$ on $X$ such that the corresponding function $G\times X\mapsto X$ defined by $(g,x)\mapsto g\cdot x$ is separately continuous. For a Banach space $X$, $\Lin(X)$ stands for the space of linear and continuous mappings from $X$ into itself. In our setting we will consider $G\subset \Lin(X)$ where $G$ can be endowed with a topology $\tau$ that could differ from the natural topology of $\Lin(X)$. A group that we will consider several times in this note is the group of permutations of $n$ elements, $\sym(\{1,\ldots, n\})$. This group induces naturally a group $G = \{T_{\sigma}: T_{\sigma}(x_1, \ldots, x_{n} ) = (x_{\sigma(1)}, \ldots ,x_{\sigma(n)} ) : \sigma \in \sym(\{1,\ldots, n\}) \}\subset \Lin (\R^n)$ (or $\Lin (\C^n)$).
	
	Given a group $G$ of $\Lin (X)$, we say that a polynomial $P$ on a Banach space $X$ is an \emph{invariant polynomial under the action of $G$} or a \emph{$G$-invariant polynomial} if $P(z) = P(\gamma(z))$ for all $z \in X$ and $\gamma \in G$. Also, we say that a polynomial $P$ on a Banach space $X$ \emph{separates a point $z$ and a set $K$ in $X$} if 
	\[
	\sup_{w \in K} |P(w)| < |P(z)|.
	\]	
	For details on polynomials and holomorphic functions on Banach spaces, see \cite{MaGarSeDe,Dineen}.
	
	Recall that the classical Hahn-Banach separation theorem provides a separation theorem by polynomials of degree one, that is, if $K$ is a nonempty closed convex balanced set in $X$ and $z$ is an element in $X \setminus K$, then there exists a linear functional $f$ on $X$ that separates $z$ and $K$.	
	 R. M. Aron et al \cite[Theorem 2.3]{AFM} showed that when $X$ is the $n$-dimensional space $\C^n$ and $G$ is a finite group of $\Lin (\C^n )$, if $K$ is a subset of $\C^n$ which is invariant under the action of $G$ and $z$ is an element in $\C^n \setminus K$ that is separated from $K$ by a polynomial $Q$, then there exists a $G$-invariant polynomial $P$ that separates $z$ and $K$. The proof makes use of a `symmetrization' $P$ of the polynomial $Q$ on $\C^n$. That is, the polynomial $P$ defined by 
	 \begin{equation}\label{def:symmetrization:finite_dimension}
	 P(w) = \sum_{\gamma \in G} Q(\gamma(w)) \quad (w \in \C^n)
	 \end{equation}
	 that is a $G$-invariant polynomial. 
	As a direct corollary it is obtained that if $K$ is a $G$-invariant closed convex balanced subset of $\C^n$ and $z$ is an element in $\C^n \setminus K$, then there exists a homogeneous $G$-invariant polynomial $P$ that separates $z$ and $K$.
	
	Coming back to general Banach spaces $X$, let $U$ be an open subset of $X$ and denote by $\mathcal{H} (U)$ an algebra of holomorphic functions defined on $U$. Suppose that $G$ is a group of $\Lin (X)$ leaving $U$ fixed, i.e., $\gamma (U) \subset U$ for every $\gamma \in G$. Then we can consider the subalgebra $\mathcal{H}_{G} (U)$ of $\mathcal{H}(U)$ that consists of the $G$-invariant holomorphic functions:
	\[
	\mathcal{H}_{G} (U) = \{ f \in \mathcal{H} (U) : h \circ \gamma = h \,\, \text{for every} \,\, \gamma \in G\}. 
	\]
	Algebras of $G$-invariant holomorphic functions are studied in \cite{AAGZ, AGPZ, CGZ}. For a compact group $G$ of $\Lin (X)$, it is shown \cite[Theorem 2.3]{AGPZ} that the mapping $S_G : \mathcal{H} (U) \rightarrow \mathcal{H}_{G} (U)$ defined as
	\begin{equation}\label{def:symmetrization:infinite_dimension}
	S_G (f) (u) = \int_G (f \circ \gamma) (u) \, d\mu_G (\gamma) \quad (u \in U), 
	\end{equation}
	where $\mu_G$ is the normalized Haar measure on $G$, is a continuous linear projection. Notice that, when $X = \C^n$ and $G$ is a finite group, then the polynomials $P$ in \eqref{def:symmetrization:finite_dimension} and $S_G(Q)$ from \eqref{def:symmetrization:infinite_dimension} coincide.
We will use variants of this symmetrization operator to obtain several separation results.	
	
	The structure of this paper is as follows. In Sections \ref{sec:real} and \ref{sec:complex} we consider the scenario of a compact topological group $G$. In Section \ref{sec:real} we show the generalisation of \cite[Theorem 2.3]{AFM} in the case of real Banach spaces: if $X$ is a real Banach space and a point can be separated from a $G$-invariant set in $X$ by a continuous polynomial on $X$, then this point can also be separated from the set by a $G$-invariant continuous polynomial on $X$. Section \ref{sec:complex} is devoted to the study of the previous result when the underlying space is a complex space, in which case the result does not always hold. We present a natural condition on the orbit induced from a given compact topological group $G$ and a continuous polynomial $Q$ in order to get the separation theorem for complex Banach spaces. 
	 We also consider the case in which the underlying space has a Schauder basis. The existence of such a basis allows us to obtain a separation theorem under natural conditions, see Theorem \ref{separation_schauder} Corollary \ref{S-separation-from-the-ball} and Corollary \ref{R-separation-from-the-ball}. In the last section we study some classical and new examples of topological groups acting on Banach spaces and we characterize points which can be separated by a group invariant polynomial from the closed unit ball in specific situations.

	\section{Real case}
	
This section is devoted to study the case when $X$ is a real Banach space and $G$ is a compact topological group. In this scenario we can obtain the following general result.	
	\label{sec:real}
	\begin{theorem}
	\label{Thrm:real-case}
		Let $(G,\tau)$ be a compact topological group of $\Lin (X)$ that acts separately continuously on $X$ and $K$ a set in $X$ that is invariant under the action of $G$. If $z$ is an element in $X \setminus K$ that can be separated from $K$ by a continuous polynomial $Q$, then there exists a $G$-invariant continuous polynomial $P$ that separates $z$ and $K$. Furthermore, if $Q$ is homogeneous, then $P$ can be chosen to be homogeneous. 
	\end{theorem}

	\begin{proof}
		Without loss of generality, we may assume that 
		\[
		\sup_{w\in K} |Q(w)| \leq r < 1 \quad \text{and} \quad |Q(z)| > 1.
		\]
		As the function $\gamma \in G \mapsto |(Q\circ \gamma )(z)| \in [0,+\infty)$ is continuous, we choose an open neighborhood $V_z$ of $Id_X$ in $G$ such that $|(Q\circ \gamma)(z)| > 1$ for every $\gamma \in V_z$. 
		
		For each $m \in \N$, consider 
		\begin{equation}\label{mth_symmetrization}
		P_m (w) = \int_G [(Q\circ \gamma ) (w)]^{m} \, d\mu (\gamma) \quad (w \in X),
		\end{equation}
		where $d\mu$ is the Haar measure on the compact group $G$. Note that $P_m$ is $mk$-homogeneous when $Q$ is $k$-homogeneous. For $w \in K$, we have that 
		\[
		|P_{2m} (w)| = \left| \int_G [(Q\circ \gamma ) (w)]^{2m} \, d\mu (\gamma) \right| \leq r^{2m}, 
		\] 
		which implies that $\sup_{w\in K} |P_{2m} (w)| \rightarrow 0$ as $m \rightarrow \infty$.

		However, we have that 
		\[
		|P_{2m} (z)| = \int_G \left[(Q\circ \gamma ) (z)\right]^{2m} \, d\mu (\gamma) \geq \int_{V_z} \left[ (Q\circ \gamma ) (z) \right]^{2m} \, d\mu (\gamma) \geq \mu(V_z)
		\]
		and it is clear by the definition of the Haar measure that $\mu(V_z)>0$. 	
	\end{proof}

	\begin{cor}
	Let $(G,\tau)$ be a compact topological group of $\Lin (X)$ that acts separately continuously on $X$ and $K$ a closed convex balanced subset of $X$ that is invariant under the action of $G$. If $z$ is an element in $X \setminus K$, then there exists a homogeneous continuous $G$-invariant polynomial that separates $z$ and $K$. 
	\end{cor}
	
	\begin{proof}
	By the Hahn-Banach separation theorem there exists a linear functional $f \in X^*$ such that $\sup_{w \in K} \text{Re}f(w) < \text{Re} f(z)$. As $K$ is balanced, we deduce that 
	\[
	\sup_{w \in K} |f(w)| < |f(z)|. 
	\]
	By Theorem \ref{Thrm:real-case}, there exists a homogeneous $G$-invariant polynomial $P$ that separates $z$ and $K$. 
	\end{proof} 

			We show now that letting $m$ tend to infinity in the proof of Theorem \ref{Thrm:real-case} is necessary to get the result. Specifically, for given $m \in \N$. The polynomial in \eqref{mth_symmetrization} is the \textit{$m$-th $G$-symmetrization of the polynomial $P$}. The following proposition shows that for a fixed natural number $N$ it is possible to find a space $X$, a group $G$ and a polynomial $Q$ which separates a given point $z \in X \setminus K$ and a set $K\subset X$, but its $l$-th $G$-symmetrizations $P_{l}$ ($1\leq l < N$) fail to separate the point $z \in X \setminus K$ and the set $K$ while the $N$-th $G$-symmetrization $P_N$ separates $z$ and $K$.

	\begin{prop}
		For every natural number $N>1$ and a natural number $k$, given $X=\mathbb R^{2k+1}$, there exist a set $K\subset X$, a point $z\in X\setminus K$ and a continuous polynomial $Q$ such that $Q$ separates $z$ and $K$, $P_{l}$ does not separate $z$ and $K$ for $l=1,\ldots,N-1$ and $P_{N}$ separates $z$ and $K$, where $P_{l}$ is the $l$-th symmetrization of $Q$ for the group $\sym(\{1,\dots, 2k+1\})$.
	\end{prop}
	\begin{proof}
		Fix a natural number $N>1$ and $n=2k+1$ and denote by $G=\sym(\{1,\dots, n\})$. Consider $X=\ell_{2N}^{n}$ and the compact set $K = \{ x= (x_1, \dots, x_n) \in \R^n : \|x\|_{\ell_{2N}^{n}} = \left( x_1^{2N} + \dots + x_n^{2N} \right)^{\frac{1}{2N}} \leq 1\}$. Let us pick $\eps,\alpha >0$ so small that 
		\begin{equation}\label{eps:condition1}
		0< \eps < n^{\frac{1}{2(N-1)}-\frac{1}{2N}} - 1
		\end{equation}
		and 
		\begin{equation}\label{alpha:condition1}
		0< \alpha < \min\left\{\frac{\eps}{2}\left(\frac{1}{n-1}\right)^{\frac{2N-1}{2N}}, \min_{l = 1, \dots, N-1}\left\{\left[ \left(\frac{n}{n-1}\right) \left\{	\left(\frac{1}{(1+\eps) \sqrt[\leftroot{-2}\uproot{2} 2N]{n}} \right)^{2l} - \frac{1}{n} \right\} 	\right]^{\frac{1}{2l}}\right\}\right\}.
		\end{equation}
		Notice that the right hand side of \eqref{alpha:condition1} is positive due to the condition \eqref{eps:condition1} on $\eps >0$. 
		
		Let us fix the point $z=(1+\eps , 0, \dots, 0)\in \mathbb R^{n}$ and the polynomial 
		\[
		Q(x_1,\dots,x_n) = \left(x_1 - \alpha (x_2 + \dots + x_n)\right)^{2}.
		\]
		We claim that the point $z$, the set $K$ and the polynomial $Q$ satisfy the first two conditions of the proposition.
	
		First we show that $Q$ separates $z$ from $K$. Note that, by H\"older's inequality,
		\begin{align*}
		|Q(x_1,\dots, x_n)| &= \left| x_1 - \alpha (x_2 + \dots + x_n) \right|^{2}\\
		 &\leq \left(1 + \alpha \|(x_2,\dots, x_n)\|_{\ell_{2N}^{n-1}} \|(1,\dots, 1)\|_{\ell_{\frac{2N}{2N-1}}^{n-1}} \right)^{2}\\
		&\leq \left(1 + \alpha (n-1)^{\frac{2N-1}{2N}} \right)^{2}\\
		&< \left(1 + \frac{\eps}{2} \left(\frac{1}{n-1}\right)^{\frac{2N-1}{2N}} (n-1)^{\frac{2N-1}{2N}}\right)^{2}\quad (\text{by }\eqref{alpha:condition1})\\
		&= \left(1+ \frac{\eps}{2}\right)^{2}
		\end{align*}
		for every $(x_1,\dots, x_n) \in K$.
		However, we have that $Q(1+\eps, 0, \dots, 0) = (1 + \eps)^{2}$.

		Now we are going to show that $P_{l}$ does not separate $z$ and $K$ for $l=1,\ldots,N-1$. Note that the $G$-invariant polynomial $P_{l}$ is given by 
		\begin{align*}
		P_{l} (x_1,\dots, x_n) 
		&= \frac{1}{n!} \sum_{\sigma\in\sym(\{1,\dots, n\})} Q(x_{\sigma(1)}, \dots, x_{\sigma(n)})^{2l} \\
		&= \frac{1}{n!} \sum_{\sigma\in\sym(\{1,\dots, n\})} \left[ x_{\sigma(1)} - \alpha(x_{\sigma(2)} + \dots + x_{\sigma(n)}) \right]^{2l}.
		\end{align*}

		From this, we can see that, as $n$ is odd, if we write $n = 2k + 1$ for some $k \in \N$, then 
		\begin{align*}
		P_{l} &\left( \frac{1}{\sqrt[\leftroot{-2}\uproot{2} 2N]{n}}, \underbrace{	\frac{1}{\sqrt[\leftroot{-2}\uproot{2} 2N]{n}}, \frac{-1}{\sqrt[\leftroot{-2}\uproot{2} 2N]{n}}, \dots, \frac{1}{\sqrt[\leftroot{-2}\uproot{2} 2N]{n}}, \frac{-1}{\sqrt[\leftroot{-2}\uproot{2} 2N]{n}}}_{2k-\text{terms}} \right) \\
		&= \frac{1}{n!} \left[\left(\frac{n+1}{2}\right)(n-1)! \left(\frac{1}{\sqrt[\leftroot{-2}\uproot{2} 2N]{n}} \right)^{2l} + \left(\frac{n-1}{2} \right)(n-1)! \left( \frac{1}{\sqrt[\leftroot{-2}\uproot{2} 2N]{n}} + \frac{2\alpha}{\sqrt[\leftroot{-2}\uproot{2} 2N]{n}} \right)^{2l} \right] \\
		&= \left(	\frac{n+1}{2n} \right) \left(\frac{1}{\sqrt[\leftroot{-2}\uproot{2} 2N]{n}}\right)^{2l} + \left(\frac{n-1}{2n} \right) \left( \frac{1}{\sqrt[\leftroot{-2}\uproot{2} 2N]{n}} + \frac{2 \alpha}{\sqrt[\leftroot{-2}\uproot{2} 2N]{n}} \right)^{2l}\\
 &> \left(	\frac{n+1}{2n} \right) \left(\frac{1}{\sqrt[\leftroot{-2}\uproot{2} 2N]{n}}\right)^{2l} + \left(\frac{n-1}{2n} \right) \left( \frac{1}{\sqrt[\leftroot{-2}\uproot{2} 2N]{n}}\right)^{2l} \\
		&= \left(\frac{1}{\sqrt[\leftroot{-2}\uproot{2} 2N]{n}}\right)^{2l}.
		\end{align*}

		However,
		\begin{align*}
		P_{l} (1+\eps, 0, 0, \dots , 0) &= \frac{1}{n!} \sum_{\substack{\sigma\in\sym(\{1,\dots, n\})\\\sigma(1)=1}} \left(1+\eps\right)^{2l} + \frac{1}{n!} \sum_{\substack{\sigma\in\sym(\{1,\dots, n\})\\\sigma(1)\ne 1}}\left(\alpha(1+\eps)\right)^{2l}\\
		&= (1+\eps)^{2l} \left(\frac{1}{n} + \left(\frac{n-1}{n} \right) \alpha^{2l} \right)\\
		&< (1+\eps)^{2l} \left(\frac{1}{(1+\eps) \sqrt[\leftroot{-2}\uproot{2} 2N]{n}} \right)^{2l} \quad(\text{by } \eqref{alpha:condition1})\\
		 &= \left(\frac{1}{\sqrt[\leftroot{-2}\uproot{2} 2N]{n}}\right)^{2l} \\
		&= P_{l} \left( \frac{1}{\sqrt[\leftroot{-2}\uproot{2} 2N]{n}}, 	\frac{1}{\sqrt[\leftroot{-2}\uproot{2} 2N]{n}}, \frac{-1}{\sqrt[\leftroot{-2}\uproot{2} 2N]{n}}, \dots, \frac{1}{\sqrt[\leftroot{-2}\uproot{2} 2N]{n}}, \frac{-1}{\sqrt[\leftroot{-2}\uproot{2} 2N]{n}} \right).
		\end{align*} 
		This shows that the $G$-invariant polynomial $P_{l}$ does not separate the point $z$ and the set $K$ for every $1 \leq l \leq N-1$.
		
		It only lasts to show that $P_N$ separates $z$ and $K$. For this, note that 
		\begin{align*}
		\lim_{\alpha\to 0}\sup_{(x_{1},\ldots,x_{n})\in K} &\left\{\frac{1}{n!} \sum_{\sigma\in\sym(\{1,\dots, n\})} \left[ x_{\sigma(1)} - \alpha(x_{\sigma(2)} + \dots + x_{\sigma(n)} \right]^{2N}\right\}\\
		&=\sup_{(x_{1},\ldots,x_{n})\in K} \left\{\frac{1}{n!} \sum_{\sigma\in\sym(\{1,\dots, n\})} x_{\sigma(1)}^{2N}\right\}
		=\frac{1}{n}.
		\end{align*}
		However, 
		\[
		P_{N} (1+\eps, 0, 0, \dots , 0) = (1+\eps)^{2N} \left(\frac{1}{n} + \left(\frac{n-1}{n} \right) \alpha^{2N} \right)>\frac{1}{n}
		\]
		for all positive numbers $\alpha$. Therefore, by choosing $\alpha$ small enough we have that $P_{N}$ separates $z$ and $K$.

	\end{proof}

	Note that in general it is not hard to find large families of mappings on a Banach space $X$ that are invariant under the action of a compact topological group $G$. Indeed, for any mapping $f$ on a Banach space $X$ the $G$-symetrization of $f$ given by
	\[
		F(x):= \int_G (f\circ \gamma ) (w) \, d\mu (\gamma)
	\] 
	is a $G$-invariant mapping. 
	
	 We would like to end this section by showing a natural way of constructing sets that are invariant under the action of a compact topological group $G$. For this, note that if $\{f_{i}\}_{i\in I}$ is a family of $G$-invariant mappings from $X$ to $\mathbb R$, the mapping $F:X\mapsto \mathbb R^{I}$ defined by $F(x)=(f_{i}(x))_{i\in I}$ is a $G$-invariant mapping. It is easily checked that the set $F^{-1}(K)$ is $G$-invariant for any set $K\subset \mathbb R^{I}$.

	\section{Complex case}
		\label{sec:complex}


This section is devoted to study the case when $X$ is a complex Banach space. However, as we will see in Example \ref{cannot-separate} we cannot expect a general result as in the real case.
The following lemma will be used several times in this note.

\begin{lemma}[{\cite[Lemma 2.2]{AFM}}]
\label{lem_AFM}
Given $r$ complex numbers of modulus one, $\{z_{1}\ldots,z_{r}\}$, there exists a strictly increasing sequence of natural numbers $(m_{k})_{k=1}^{\infty}$ such that the sequence $(z_{j}^{m_{k}})$ converges to $1$ for $j=1,\ldots,r$.
\end{lemma}

We denote the argument of a complex number $z$ by $Arg(z)$.

\begin{theorem}\label{complex_separation}
	Let $(G,\tau)$ be a compact topological group of $\Lin (X)$ that acts separately continuously on a complex Banach space $X$ and $K$ a set in $X$ that is invariant under the action of $G$. If $z$ is an element in $X \setminus K$ that can be separated from $K$ by a polynomial $Q$ which satisfies that $\vert Q(z)\vert>1$ and for some $\eta \in (0,1)$ 
	\begin{equation}\label{condition_complex_separation}
		\# \{ Arg(Q(\gamma(z))) \in [0,2\pi) : |Q(\gamma(z))| \geq \eta, \, \gamma \in G\} \,\, \text{is finite},
	\end{equation}	
	then there exists a $G$-invariant continuous polynomial $P$ that separates $z$ and $K$. Furthermore, if $Q$ is homogeneous, then $P$ can be chosen to be homogeneous. 
\end{theorem}

\begin{proof}
	Without loss of generality, we may assume that 
	\[
	\sup_{w\in K} |Q(w)| \leq r < 1 \quad \text{and} \quad |Q(z)| > 1.
	\]
	As the function $\gamma \in G \mapsto |(Q\circ \gamma )(z)| \in [0,+\infty)$ is continuous, we choose an open neighborhood $V_z$ of $Id_X$ in $G$ such that $|(Q\circ \gamma)(z)| > 1$ for every $\gamma \in V_z$. 

	For each $m \in \N$, consider 
	\[
	P_m (w) = \int_G [(Q\circ \gamma ) (w)]^m \, d\mu (\gamma) \quad (w \in X),
	\]
	where $d\mu$ is the Haar measure on the compact group $G$. Note that $P_m$ is $mk$-homogeneous when $Q$ is $k$-homogeneous. 
	
	On the one hand, for $w \in K$, we have that 
	\[
	|P_m (w)| = \left| \int_G [(Q\circ \gamma ) (w)]^m \, d\mu (\gamma) \right| \leq r^m, 
	\] 
	which implies that $\sup_{w\in K} |P_m (w)| \rightarrow 0$ as $m \rightarrow \infty$.
	
	On the other hand, for each $\gamma \in G$, let $\theta_{\gamma}$ be a real number such that $Q(\gamma(z)) = |Q(\gamma(z))|e^{i\theta_{\gamma}}$. If we let 
	$G_{\eta} := \{ \gamma \in G : |Q(\gamma(z))| \geq \eta \}$, then by Lemma \ref{lem_AFM} there exists a natural number $m$ large enough so that 
	\[
	\left| e^{i\theta_{\gamma} m} - 1 \right| < r, \quad \forall\, \gamma \in G_{\eta}.
	\] 
	Now, we have that 
	\begin{align*}
	|P_{m} (z)| &= \left| \int_G [(Q\circ \gamma ) (z)]^{m} \, d\mu (\gamma) \right| \\
	&= \left| \int_G \left|(Q\circ \gamma ) (z)\right|^{m} e^{i \theta_{\gamma} m} \, d\mu (\gamma) \right| \\
	&\geq \int_G \left|(Q\circ \gamma ) (z)\right|^{m} \, d\mu (\gamma) - \int_G \left|e^{i\theta_{\gamma} m} - 1\right| \left|(Q\circ \gamma ) (z)\right|^{m} \, d\mu (\gamma) 
	\end{align*}
	However, 
	\begin{align*}
	\int_G \left|e^{i\theta_{\gamma} m} - 1\right| \left|(Q\circ \gamma ) (z)\right|^{m} \, d\mu (\gamma) 
	&\leq \int_{G_{\eta}} \left|e^{i\theta_{\gamma} m} - 1\right| \left|(Q\circ \gamma ) (z)\right|^{m} \, d\mu (\gamma) + 2\eta^m \\
	&\leq r \int_{G_{\eta}} \left|(Q\circ \gamma ) (z)\right|^{m} \, d\mu (\gamma) + 2 \eta^m\\
	&\leq r \int_{G} \left|(Q\circ \gamma ) (z)\right|^{m} \, d\mu (\gamma) + 2 \eta^m. 
	\end{align*}
		
It follows that for $m \in \N$ big enough 
	\begin{align*}
	|P_{m} (z)| &\geq \int_G \left|(Q\circ \gamma ) (z)\right|^{m} \, d\mu (\gamma) - \left( r \int_{G} \left|(Q\circ \gamma ) (z)\right|^{m} \, d\mu (\gamma) + 2\eta^m \right) \\
	&= (1-r ) \int_G \left|(Q\circ \gamma ) (z)\right|^{m} \, d\mu (\gamma) - 2\eta^m \\
	&\geq (1-r ) \int_{V_{z}} \left|(Q\circ \gamma ) (z)\right|^{m} \, d\mu (\gamma) - 2\eta^m \geq (1-r )\mu(V_{z}) - 2\eta^m > \frac{1}{2}(1-r)\mu(V_{z})> 0
	\end{align*}
	as $\mu(V_z) > 0$ clearly by definition of the Haar measure.
	
\end{proof}

	\begin{rem}
	Note that if, in the previous theorem, the separating continuous polynomial $Q$ is already $G$-invariant then condition \eqref{condition_complex_separation} is trivially satisfied since the set consists of only one point, $Q(z)$. 
	\end{rem}

 \begin{rem}
	It is clear that if $G$ is a finite group, then the assumption \eqref{condition_complex_separation} holds for any continuous polynomial $Q$, which implies that Theorem \ref{complex_separation} is a generalisation of \cite[Theorem 2.3]{AFM}.
	However, even if the assumption \eqref{condition_complex_separation} is satisfied for some continuous polynomial $Q$, it does not imply the finiteness of the cardinality of $G$. For example, consider the group $R$ of operators on $c_0$ given by 
	\[
	\gamma (z) = (z_1,e^{2\pi ik_{2}/2} z_{2}, \ldots, e^{2\pi ik_{m}/m} z_m, \ldots ) \quad (k_{2},k_{3}, \ldots\in \mathbb N\cup\{0\}),
	\]
which is compact with the weak operator topology of $\Lin (c_0)$.
	For given $\eps > 0$ and $z \in c_0$ which lies outside the closed unit ball $\overline{B_{c_0}}$, there always exist $N, M \in \N$ with $N < M$ such that $|z_N| > 1$ and $|z_n| < \eps$ for every $n \geq M$. This implies that the coefficient functional $e_N^*$ separates $z$ from $\overline{B_{c_0}}$. Since $|e_N^* (\gamma (z)) | = |e_N^* (z)| > 1$ for every $ \gamma \in R$, 
	\begin{align*}
	\{ e_N^* ( \gamma(z) ) \in \C : | e_N^* ( \gamma(z) )| 
		\geq \eps, \, \gamma \in R \} 
	 	&= \{z_N, e^{2\pi i / N} z_N, \ldots, e^{2\pi i (2N-1) / N} z_N \};
	\end{align*}
	hence the assumption \eqref{condition_complex_separation} holds for the polynomial $e_N^*$, while $R$ is an infinite group. Furthermore, the number of elements in the set in \eqref{condition_complex_separation} can be made arbitrarily large if we modify the point $z$. We will study separation theorems associated to this group in more detail in Section \ref{sec:roots}.
	\end{rem}

As in the previous section, we would like to show a natural way of constructing sets that are invariant under the action of compact topological group. Clearly, the sets $F^{-1}(K)$ constructed from a family $(f_{i})_{i\in I}$ of $G$-invariant sets, i.e., $F(x) = (f_i(x))_{i\in I} : X \rightarrow \C^I$, will produce $G$-invariant sets. But we are more interested in the case where $I$ is finite. In this case, the definition of a polynomially convex set plays an important role since automatically the sets that we obtain satisfy that any point outside of the set can separated from the set by a $G$-invariant continuous polynomial. We recall now the definition of a polynomially convex set. The \textit{polynomially convex hull} of a compact subset $K$ of $\mathbb C^{n}$ is the set
 $$\widehat K = \{z \in \mathbb C^{n} : \vert P (z)\vert \leq \sup_{w \in K}|P(w)| \text{ for every polynomial } P \}.
 $$
A compact set $K$ of $\mathbb C^{n}$ is \textit{polynomially convex} if $\widehat K=K$.

\begin{prop}
Let $(G,\tau)$ be a compact topological group of $\Lin (X)$ that acts separately continuously on a complex Banach space $X$ and $f_{1},\ldots, f_{n}$ be continuous $G$-invariant polynomials on $X$. Let $F(x)=(f_{1}(x),\ldots,f_{n}(x))$. For any polynomially convex set $K\subset \mathbb C^{n}$ the set $F^{-1}(K)$ is $G$-invariant and any point $z\notin F^{-1}(K)$ can be separated from $F^{-1}(K)$ by a continuous $G$-invariant polynomials.
\end{prop}

In general, we cannot expect in the complex case a general result like Theorem \ref{Thrm:real-case}, even in the case when $X$ is an $1$-dimensional space, as the following example shows. 

\begin{example}
	\label{cannot-separate}
	Let $\mathbb{T} = \{z \in \C : |z| =1 \}$ be the circle group (which is compact) that acts continuously on $\mathbb C$ if we consider multiplication of the complex numbers. Then it is obvious that $\mathbb{D} \subset \C$ is invariant under the action of $G = \{ w \mapsto zw : z \in \mathbb{T}\}$. Let us consider $f(z) := z$ for every $z \in \C$. Then every $w \in \C \setminus \mathbb{D}$ is separated from $\mathbb{D}$ by $f$. However, polynomials which are invariant under $G$ are just constant ones and a constant function cannot separate $w$ from $\mathbb{D}$. 
\end{example}

The above example shows that some conditions are required both for the space and for the group in order to have positive separation results. This is why we focus now on the case where the Banach space $X$ has a Schauder basis $(e_n)_{n\in\N}$. We denote by $\pi_n$ the restriction on the first $n$ coordinates, that is, if $z = \sum_{j=1}^{\infty} z_j e_j$, then $\pi_n(z)=\sum_{j=1}^{n} z_j e_j$ Let us denote by $\Pi_n : X \rightarrow \C^n$ the projection defined as $\Pi_n (z) = (z_1, \ldots, z_n)$ for every $z = \sum_{j=1}^{\infty} z_j e_j$ in $X$ and by $\iota_n : \C^n \rightarrow X$ the inclusion defined as $\iota_n (z_1,\ldots,z_n) = \sum_{j=1}^{n} z_j e_j $. Observe that $\pi_n = \iota_n \circ \Pi_n$ for each $n \in \N$. Given a group $G \subset \Lin (X)$ and $n \in \N$, consider the following set 
\[
\Sigma_n (G) := \{ \Pi_n \circ g \circ \iota_n : g \in G\} \subset \Lin(\C^n).
\]
Note that $\Sigma_n (G)$ need not be a group. Indeed, assume that $G$ is a group such that the element $g$ defined as 
\[
g(e_j) = \begin{cases}
e_j \quad &\text{for} \,\, j \neq N, j \neq N+1, \\
e_{N+1} \quad &\text{for} \,\, j=N, \\
e_{N} \quad &\text{for} \,\, j=N+1
\end{cases}
\]
for some fixed $N \in \N$ is in $G$. Then one can easily check that 
\[
(\Pi_N \circ g \circ \iota_N) (w) = (w_1, \ldots, w_{N-1}, 0), \quad w = (w_1, \ldots, w_{N}) \in \C^N;
\]
hence we have that $\Pi_N \circ g \circ \iota_N \in \Lin(\C^N)$ is not invertible. 

However, in many natural cases it is satisfied that for a sequence of natural numbers $(j_{n})_{n=1}^{\infty}$ the set $\Sigma_{j_{n}} (G)$ is a group. This is for instance the case when $G$ is the Cartesian product of the groups $G_{1}, G_{2},\ldots$ acting each one of them on a finite dimensional Banach space $X_{1},X_{2},\ldots$ and we consider the group $G_{1}\times G_{2}\times \cdots$ and the product space, with a suitable norm, $X=X_{1}\times X_{2}\times \cdots$ and the natural action of $G$ on $X$. 

The following straightforward lemma will be used in Theorem \ref{separation_schauder}.

\begin{lemma}\label{projected_invariant_group}
	Let $X$ a complex Banach space with Schauder basis, $(G,\tau)$ be a group of $\Lin (X)$ that acts separately continuously on the complex Banach space $X$ and $K$ a $G$-invariant set in $X$. If there exists $n \in \N$ such that 	
	$\Sigma_n (G)$ is a group of $\Lin (\C^n)$ and $\pi_{n} (K) \subset K$,
	then the set $\Pi_n (K)$ in $\C^n$ is invariant under the action of $\Sigma_n (G)$.
\end{lemma}

\begin{proof}
	Let $\Pi_n \circ g \circ \iota_n \in \Pi_n (G)$ and $\Pi_n (z) \in \Pi_n (K)$ be given. As the element $\pi_n (z)$ belongs to $K$, $g (\pi_n (z)) = z'$ for some $z' \in K$; hence 
	\begin{equation*}	
	(\Pi_n \circ g \circ \iota_n) (\Pi_n (z)) = \Pi_n (g(\pi_n (k)) ) = \Pi_n (z') \in \Pi_n (K).
	\end{equation*}
	This shows that the set $\Pi_n (K)$ is a $\Sigma_n (G)$-invariant set.
\end{proof}

\begin{theorem}\label{separation_schauder}
	Let $X$ a complex Banach space with Schauder basis, $G$ be a group of $\Lin (X)$ that acts separately continuously on the complex Banach space $X$ and $K$ a $G$-invariant set in $X$. If there exists a subsequence $(j_n) \subset \N$ such that 
	\vspace{-1mm}
	\begin{enumerate}
		\setlength\itemsep{0.4em}
		\item $\Sigma_{j_n} (G)$ is a compact subgroup for each $n \in \N$,
		\item $g(\spann \{e_i : i \geq j_n +1 \} ) \subset \spann \{e_i :i \geq j_n +1 \}$ for every $n \in \N$ and $g \in G$,
		\item $\pi_{j_n} (K) \subset K$,
	\end{enumerate} 
	and if an element $z$ in $X \setminus K$ can be separated from $K$ by a continuous polynomial $Q$, which satisfies that for some $r \in (0,1)$ 
	\begin{equation}\label{condition_separation_Schuader}
	\# \{ Arg\left( (Q \circ \iota_{j_n}) (\gamma(z) )\right) \in [0,2\pi) : |(Q \circ \iota_{j_n}) (\gamma(z) )| \geq r, \, \gamma \in \Sigma_{j_n} (G) \} \,\, \text{is finite for each} \,\, n \in \N,
	\end{equation}
	 then there exists a $G$-invariant continuous polynomial $P$ that separates $z$ and $K$. Furthermore, if $Q$ is homogeneous, then $P$ can be chosen to be homogeneous. 
\end{theorem}

\begin{proof}
	Suppose that $\sup_{w\in K} |Q(w)| < r < |Q(z)|$ for some $r >0$. Choose $n \in \N$ such that $|Q( \pi_{j_n} (z))| > r$. Note from $\sup_{w \in K} |(Q\circ \pi_{j_n})(w)| < r$ that the polynomial $Q \circ \iota_{j_n}$ separates the point $\Pi_{j_n} (z) \in \C^{j_n} $ and the set $\Pi_{j_n} (K) \subset \C^{j_n}$. Note from Lemma \ref{projected_invariant_group} that the set $\Pi_{j_n} (K)$ is invariant under the action of the compact group $\Sigma_{j_n} (G)$. 
	By Theorem \ref{complex_separation}, there exists a $\Sigma_{j_n} (G)$-invariant continuous polynomial $\widetilde{P}$ (which is homogeneous whenever $Q$ is) that separates $\Pi_{j_n} (z)$ and $\Pi_{j_n} (K)$. Define $P := \widetilde{P} \circ \Pi_{j_n} $. Then $P$ separates the point $z$ and the set $K$. Observe that 
	\begin{align*}
	(P\circ g)(z) = (\widetilde{P} \circ \Pi_{j_n}) (g (z) ) &= (\widetilde{P} \circ \Pi_{j_n}) \left( (g \circ \iota_{j_n}) (\Pi_{j_n} (z)) + \sum_{k={j_n +1}}^{\infty} z_k g( e_k ) \right) \\
	&= (\widetilde{P} \circ \Pi_{j_n} \circ g \circ \iota_{j_n}) (\Pi_{j_n} (z)) \\
	&= (\widetilde{P} \circ \Pi_{j_n}) (z) \\
	&= P(z)
	\end{align*}
	for every $g \in G$ and $z = \sum_{j=1}^{\infty} z_j e_j \in X$. This implies that $P$ is a $G$-invariant continuous polynomial.
\end{proof}

It is clear that if $\Sigma_{j_n} (G)$, in the previous theorem, is a finite group for each $n \in \N$, then the condition \eqref{condition_separation_Schuader} is true automatically for any continuous polynomial $Q$. 

\begin{cor}\label{cor_separation_schauder}
			Let $X$ a complex Banach space with Schauder basis, $G$ be a group of $\Lin (X)$ that acts continuously on the complex Banach space $X$ and $K$ a $G$-invariant set in $X$. If there exists a subsequence $(j_n) \subset \N$ such that 
	\vspace{-1mm}
	\begin{enumerate}
		\setlength\itemsep{0.4em}
		\item $\Sigma_{j_n} (G)$ is a finite group for each $n \in \N$,
		\item $g(\spann \{e_i : i \geq j_n +1 \} ) \subset \spann \{e_i :i \geq j_n +1 \}$ for every $n \in \N$ and $g \in G$,
		\item $\pi_{j_n} (K) \subset K$,
	\end{enumerate} 
	and if an element $z$ in $X \setminus K$ can be separated from $K$ by a continuous polynomial $Q$, then there exists a $G$-invariant continuous polynomial $P$ that separates $z$ and $K$. Furthermore, if $Q$ is homogeneous, then $P$ can be chosen to be homogeneous. 
\end{cor}

\section{Separation theorems for classical groups}

In this section we study separation theorems for several classical groups which in general are not compact.

Given a Banach space $X$ with Schauder basis $(e_j)_{j=1}^\infty$, it can be associated to a Banach sequence space $\tilde X=\psi(X)$ where $\psi:X\mapsto \mathbb K{^\mathbb N}$ is given by $\psi(\sum_{j=1}^\infty x_je_j)=(x_j)_j$ and $\tilde x=\psi(x)$. Moreover, if we define $\Vert (x_j)_j\Vert:=\Vert \sum_j x_je_j\Vert$ then $\psi$ is an isometric isomorphism from $X$ onto $\tilde X$. For this reason we will work in this section in the more general setting of Banach sequence spaces.

\subsection{ Roots of unity.}
\label{sec:roots}

	Let $X$ be a Banach sequence space. Consider the group $R$ consisting of the operators
	\[
	\gamma (z) = (z_1,e^{2\pi ik_{2}/2} z_{2}, \ldots, e^{2\pi ik_{m}/m} z_m, \ldots ) \quad (k_{2},k_{3}, \ldots\in \mathbb N\cup\{0\}).
	\]
	for $z = (z_m)_{m\in \N}$, or the subgroup $R_{F}$ of $R$ generated by $\{\gamma_m\}_{m\in\N}$, where 
	\[
	\gamma_m (z) = (z_1, \ldots, z_{m-1}, e^{2\pi i/m} z_m, z_{m+1}, \ldots ).
	\]
	Each group can be endowed with natural topologies of $\Lin (X)$. Note that for instance if $X$ is the Banach space $c_0$, then the group $R$ is compact in the weak operator topology of $\Lin (c_0)$. However neither $R$ nor the subgroup $R_F$ are compact in the strong operator topology of $\Lin (c_0)$ (so, in the norm topology of $\Lin (c_0)$). 
		
	 The action of the subgroup $R_{F}$ on the spaces $\mathbb C^{n}$ and $c_{0}$ was studied in \cite[Section 4]{AGPZ} where the authors show that any $k$-homogeneous polynomial $Q$ on $\mathbb C^{n}$ that is invariant under the action of $R_F$ can be written as $Q(z_{1},z_{2},\ldots,z_{n})=\tilde Q(z_{1},z_{2}^{2},\ldots,z_{n}^{n})$ where $\tilde Q$ is a polynomial defined on $\mathbb C^{n}$. The analogous happens on $c_{0}$ where \cite[Corollary 4.3.]{AGPZ} shows that any $k$-homogeneous $R_F$-invariant polynomial $Q$ can be written as $Q(z_{1},z_{2},\ldots,z_{n}, z_{n+1},\ldots)=\tilde Q(z_{1},z_{2}^{2},\ldots,z_{n}^{n})$ for some natural number $n$ and $\tilde Q$ is a polynomial defined on $\mathbb C^{n}$.
		
	Assume that $R$ acts continuously on a space $X$ with 1-unconditional Schauder basis, such as $c_0$, $\ell_p$ (for $1\leq p<\infty$), Lorentz sequence spaces or Orlicz sequence spaces. Let us denote by $\widetilde{R}$ the group generated by $\tilde{\gamma} : \C^n \rightarrow \C^n$ defined as 
	\[
	\tilde{\gamma} (w_1, \ldots, w_n) = (w_1, e^{2\pi ik_{2}/2} w_{2}, \ldots, e^{2\pi i k_{n}/n} w_n ), \quad (w_1,\ldots,w_n) \in \C^n 
	\] 
	for some $k_2, \ldots, k_n \in \N \cup \{0\}$. If $\gamma \in R$ is such that 
	$\gamma (z) = (z_1,e^{2\pi ik_{2}/2} z_{2}, \ldots, e^{2\pi ik_{m}/m} z_m, \ldots )$, then
	\[
	(\Pi_n \circ \gamma \circ \iota_n) (w) = \tilde{\gamma} (w)	
	\]
	for every $w=(w_1,\ldots,w_n) \in \C^n$. This shows that $\Sigma_n (R)$ is a finite group contained in $\Lin (\C^n)$ for every $n \in \N$. Moreover, it is clear that $\gamma (\spann \{ e_i : i \in I\}) \subset \spann \{ e_i : i \in I\}$ for every $\gamma \in R$ and $I \subset \N$. Thus, the group $R$ and its subgroup $R_F$ satisfy the conditions in Corollary \ref{cor_separation_schauder}, so we have the following result.

\begin{prop} 
Let $X$ be a Banach sequence space such that its canonical basis is 1-unconditional Schauder basis. If a set $K$ is invariant under the group $R$ (resp. $R_{F}$) with $\pi_{j_n} (K) \subset K$ for some subsequence $(j_n) \subset \N$ and $z$ is an element in $X \setminus K$ that can be separated from $K$ by a continuous polynomial $Q$, then there exists a $R$-invariant (resp. $R_{F}$-invariant) continuous polynomial $P$ that separates $z$ and $K$. Furthermore, if $Q$ is homogeneous, then $P$ can be chosen to be homogeneous. 	
\end{prop}

The natural $R$-invariant set to consider on a Banach sequence space is the unit ball of the space, for which the following separation theorem holds.

\begin{cor}
\label{S-separation-from-the-ball}
	 Let $X$ be a Banach sequence space such that its canonical basis is 1-unconditional Schauder basis such that $\overline{B_X}$ is $R$-invariant (or $R_{F}$-invariant). If $z$ is an element outside $\overline{B_X}$, then $z$ can be separated from $\overline{B_X}$ by a homogeneous $R$-invariant (or $R_{F}$-invariant) continuous polynomial. 
\end{cor}

\subsection{Groups of permutations}
\label{permutations}
The set of polynomials invariant under the group of permutations of the natural numbers has been intensively studied on the spaces $c_{0}$ and $\ell_{p}$ for $1 \leq p < \infty$ and in a more general setting on Banach sequence spaces, see for instance \cite{CMS,Gar,GGJ,NS} and the references therein.
Let us denote by $\mathcal{G}$ the group of all permutations on the natural numbers. Among other results, it is observed in \cite[Theorem 1.1]{GGJ} that the polynomials $F_k$ on $\ell_p$ given by 
\begin{equation}
\label{basicpolynomials}
F_k (x) = \sum_{j=1}^{\infty} x_j^k \quad (x=(x_j)_{j=1}^{\infty} \in \ell_p),
\end{equation}
for $k = \lceil p \rceil, \lceil p \rceil +1, \ldots$, where $\lceil p \rceil$ is the smallest integer that is greater than or equal than $p$, form an algebraic basis of the algebra of all $\mathcal{G}$-invariant polynomials on $\ell_p$. Such polynomials $F_k$ are called \emph{elementary $\mathcal{G}$-invariant polynomials}. Moreover, a polynomial on the space $\ell_p$ for $1 \leq p < \infty$, is $\mathcal{G}$-invariant if and only if it is symmetric with respect to the subgroup $\mathcal{G}_0:= \bigcup_{n\in\N} \sym (\{1,\ldots, n\})$, where $\sym (\{1,\ldots, n\})$ denotes the set of permutations $\sigma$ of the natural numbers such that $\sigma(j)=j$ for $j>n$ (see \cite{GGJ}).

By using similar arguments to the ones used in the proof of Theorem \ref{complex_separation} we can obtain the following result. 

\begin{prop}\label{prop:permutations:ell_p}
Set $1 \leq p < \infty$. If $z=(z_j)_{j=1}^{\infty}\in \ell_p$ is such that 
 for some $j_0\in\mathbb N$, $\vert z_{j_0}\vert>1$, then $z$ can be separated from $\overline{B_{\ell_p}}$ by an elementary $\mathcal{G}$-invariant polynomial.
\end{prop}

\begin{proof} 
Since $z\in \ell_p$, there exists a natural number $N\geq j_0$ such that $\sum_{j>N}\vert z\vert^p<1$. Let $\theta_j \in [0,2\pi)$ be so that $z_j = |z_j| e^{i \theta_j}$ for each $j =1,\ldots, N$.
Using Lemma \ref{lem_AFM}, we may choose an increasing sequence of natural numbers $(n_k)_{k=1}^\infty$ such that $\sup_{j=1,\ldots,N} |e^{i \theta_j n_k} - 1 | <\frac{1}{2}$ for every $k \in \N$. Observe that 
\begin{align*}
|F_{n_k} (z) | = \left| \sum_{j=1}^{\infty} z_j^{n_k} \right| &= \left| \sum_{j=1}^{N} z_j^{n_k} +  \sum_{j>N} z_j^{n_k} \right| \\
&\geq \sum_{j=1}^{N} |z_j|^{n_k} - \sup_{j=1,\ldots,N} |1- e^{i \theta_j n_k}| \sum_{j=1}^{N} |z_j|^{n_k} - \sum_{j > N} |z_j|^{n_k}.
\end{align*} 
For sufficiently large $n_k$ and bigger than $p$, we have 
\begin{align*}
|F_{n_k} (z) | 
 \geq \frac{1}{2}|z_{j_0}|^{n_k} -\sum_{j>N} |z_j|^{n_k} \geq \frac{1}{2}|z_{j_0}|^{n_k} - 1\xrightarrow[k \rightarrow \infty]{} \infty. 
\end{align*} 
As $|F_{n_k} (x)| \leq 1$ for every $x \in \overline{B_{\ell_p}}$, we complete the proof. 
\end{proof}

\begin{rem}
If $z=(z_j)_j\notin \overline{B_{\ell_p}}$ but $\vert z_j\vert\leq 1$ for every $j \in \N$, nothing can be said in general. That is, the point $z$ can be or cannot be separated by an elementary $\mathcal{G}$-invariant polynomial. For example, if $p \in \N$ and $z_j\geq0$ for every $j\in \N$, then taking $F_p (x)=\sum_{n=1}^\infty x_n^p$, we will have $F_p (z) > 1 \geq \Vert x\Vert \geq F_p(x)$ for every $x\in \overline{B_{\ell_p}}$. Hence, the point $z$ can be separated from $\overline{B_{\ell_p}}$ by the elementary $\mathcal{G}$-invariant polynomial $F_p$.

On the other hand, if $p\notin \N$ there always exists $z=(z_j)_j\notin \overline{B_{\ell_p}}$, $z_j>0$ for all natural number $j$ that cannot be separated from $\overline{B_{\ell_p}}$ by any elementary $\mathcal{G}$-invariant polynomial. Indeed, let $m\in\mathbb N$ be chosen so that $m < p < m+1$, that is, $\lceil p \rceil = m+1$. Consider the sequence 
\[
z=\left(\frac{1}{2^{\frac{1}{m+1}} }, \ldots, \frac{1}{2^{\frac{n}{m+1}}}, \ldots \right) \in \ell_p;
\]
then 
\[
\|z\|_p= \frac{1}{2^{\frac{p}{m+1}} - 1} > 1.
\]
However, if $F_k(x)=\sum_{n=1}^\infty x_n^k$ with $k \geq m+1$, then 
\[
|F_k (z)| = \frac{1}{2^{ \frac{k}{m+1} } } + \cdots + \frac{1}{2^{ \frac{kn}{m+1} } } + \cdots = \frac{1}{2^{\frac{k}{m+1}} - 1} \leq 1,
\]
which implies that $z$ cannot be separated from $\overline{B_{\ell_p}}$ by elementary $\mathcal{G}$-invariant polynomials. Observe that $F_k$ is not well defined if $k<p$.

\end{rem}

However the group $\mathcal{G}$ of permutations on the natural numbers can be too restrictive in many cases. For instance it was shown in \cite[Example 1.3]{GGJ} that there are no non-null polynomials on $c_{0}$ that are $\mathcal{G}$-invariant. The same happens for several Orlicz sequence spaces and Lorentz sequence spaces, see \cite[examples 1.4 and 1.5]{GGJ}. Also, it was observed in \cite[Theorem 3.2, Theorem 5.5]{GVZ} that the only homogeneous polynomial on $\ell_{\infty}$ or on $L_{\infty} [0, +\infty)$ which is $\mathcal{G}$-invariant is the null polynomial. This motivates us to consider of a smaller group of permutations on the natural numbers.

Let $X$ be a Banach sequence space and $G$ be a subgroup of the group of permutation of the natural numbers. 
Given $\sigma \in G$, we can consider the linear operator $\widehat{\sigma}\in\Lin (X)$ defined as 
\[
\widehat{\sigma} ( x )= ( x_{\sigma(j)} )_{j=1}^{\infty} \quad (x = (x_j)_{j=1}^{\infty} \in X).
\]
Note that as the operator $\widehat{\sigma}$ is well defined, by the closed graph theorem we have that the operator is continuous. We denote by $\widehat{G}$ the set $\{ \widehat{\sigma} : \sigma \in G\}\subset\Lin(X)$.

 
 Fix a strictly increasing sequence of non-negative integers numbers $(n_k)_{n=1}^\infty$ with $n_1=0$. Set $N_k=\{n_k+1,\ldots,n_{k+1}\}$ for $k \in \N$. Consider the group $S$ of permutations of the natural numbers such that $\sigma(N_k)=N_k$ for each $k \in \N$. 
Given $n \in \N$ and $\sigma \in S$, observe that 
\[
(\Pi_{n}\circ \widehat{\sigma} \circ \iota_{n}) (w) = \left(w_{\sigma(1)}, \ldots, w_{\sigma (n)} \right).
\]
This shows that $\Sigma_{n} (\widehat{S})$ is a finite group contained in $\Lin (\C^{n})$ for each $n \in \N$ and Corollary \ref{cor_separation_schauder} yields the following result. 

\begin{prop}\label{S-separation:general_ver}
	Let $X$ be a Banach sequence space such that its canonical basis is a Schauder basis and $\widehat{S}$ defines a separately continuous action on $X$. If a set $K$ is invariant under the group $\widehat{S}$ with $\pi_{j_n} (K) \subset K$ and $z$ is an element in $X \setminus K$ that can be separated from $K$ by a continuous polynomial $Q$, then there exists a $\widehat{S}$-invariant continuous polynomial $P$ that separates $z$ and $K$. Furthermore, if $Q$ is homogeneous, then $P$ can be chosen to be homogeneous. 	
\end{prop}

As before a natural $\widehat{S}$-invariant set to consider on a Banach sequence space is the unit ball of the space. In this case, the analogous to Corollary \ref{S-separation-from-the-ball} also holds.

\begin{cor}
\label{R-separation-from-the-ball}
	 Let $X$ be a Banach sequence space such that its canonical basis is 1-unconditional Schauder basis so that $\widehat{S}$ defines an action on $X$. Assume that $\overline{B_X}$ is $\widehat{S}$-invariant. If $z$ is an element outside $\overline{B_X}$, then $z$ can be separated from $\overline{B_X}$ by a homogeneous $\widehat{S}$-invariant continuous polynomial. 
\end{cor}

Notice that due to the fact that the space $\ell_{\infty}$ has no Schauder basis, the above Proposition \ref{S-separation:general_ver} cannot be applied to $\ell_{\infty}$. Meanwhile, a permutation $\sigma$ of the natural numbers is called a \emph{finite bijection} if there is $n \in \N$ such that the restriction of $\sigma$ to the set $\{ n, n+1, \ldots \}$ is the identity map. Let us denote by $\mathcal{G}_f$ the group of all finite bijections on the natural numbers. It is observed in \cite[Theorem 4.3]{GVZ} that a bounded type entire function $f$ on $\ell_{\infty}$ is $\mathcal{G}_f$-invariant if and only if there is a bounded type entire function $\widetilde{f}$ on $\ell_{\infty} / c_0$ such that $f = \widetilde{f} \circ Q$, where $Q$ is the quotient map from $\ell_{\infty}$ to $\ell_{\infty} / c_0$.

\begin{prop}\label{Gf-separation-from-the-ball}
Let $z=(z_j)_{j=1}^{\infty}$ be an element in $\ell_{\infty}$. The point $z$ can be separated from $\overline{B_{\ell_{\infty}}}$ by a $\mathcal{G}_f$-invariant continuous polynomial if and only if 
$\limsup_{j\rightarrow \infty} |z_{j}| > 1$.
\end{prop}

\begin{proof} Put $\alpha = \limsup_{j\rightarrow \infty} |z_{j}|$ and suppose that $\alpha > 1$. Hence, there exists a sequence $(j_n)_{n\in \N} \subset \N$ such that 
\[
|z_{j_n}| \xrightarrow[n \rightarrow \infty]{} \alpha > 1. 
\]
Let $I = \{ j_n : n \in \N\}$ and $\mathcal{U}$ be a free ultrafilter on $I$. Consider the linear functional $\varphi$ on $\ell_{\infty}$ given by 
\[
\varphi (x) = \lim_{\mathcal{U}} x_i \quad (x \in \ell_{\infty}).
\]
As $\varphi (x + y) = \varphi (x)$ for every $x \in \ell_{\infty}$ and $y \in c_0$, we see that $\varphi$ is a $\mathcal{G}_f$-invariant polynomial. It is clear that $\sup_{x \in \overline{B_{\ell_{\infty}}}} |\varphi (x)| \leq 1$. However, $|\varphi (z)| = \lim_{n\rightarrow \infty} |z_{j_n}| > 1$; hence $z$ is separated from $\overline{B_{\ell_{\infty}}}$ by $\varphi$. 

Conversely, assume that $\alpha \leq 1$. Notice that if $\alpha < 1$, then there exist at most finitely many $z_n$ whose modulus greater than $1$. Say, $\{ n \in \N : |z_n | > 1 \} = \{n_1, \ldots, n_l\}$ for some $n_1, \ldots, n_l \in \N$. If $P$ is a $\mathcal{G}_f$-invariant continuous polynomial on $\ell_{\infty}$, then 
\[
P(z) = P(z - (0,\ldots,0, z_{n_1}, 0, \ldots, 0, z_{n_2}, 0, \ldots, 0, z_{n_l}, 0, \ldots )) 
\]
since $w:= (0,\ldots,0, z_{n_1}, 0, \ldots, 0, z_{n_2}, 0, \ldots, 0, z_{n_l}, 0, \ldots )$ belongs to $c_0$, see  \cite[Theorem 4.3]{GVZ}. As every component of $z-w$ has modulus less than equal to $1$, we have that 
\[
|P(z)| = |P(z-w)| \leq \sup_{x \in \overline{B_{\ell_{\infty}}}} |P(x)|.
\]
Now, assume that $\alpha = 1$. Then, for each $r \in (0,1)$, we have that 
$\limsup_{j \rightarrow \infty} |r z_j| = r \alpha = r < 1$. From the preceding observation, we have that 
\[
|P(rz)| \leq \sup_{x \in \overline{B_{\ell_{\infty}}}} |P(x)|.
\]
By letting $r \rightarrow 1$, we get that $|P(z)| \leq \sup_{x \in \overline{B_{\ell_{\infty}}}} |P(x)|$; hence we conclude that $z$ cannot be separated from $\overline{B_{\ell_{\infty}}}$ by a $\mathcal{G}_f$-invariant continuous polynomial. 
\end{proof}

	\subsection{The group of supersymmetries.} 
	Let us fix a finite or infinite subset $A$ of the natural numbers. We denote by $A_{0}=A\cup -A\subset \mathbb Z\setminus \{0\}$ and by $\ell_p(A_{0})$, $1\leq p\leq \infty$, the Banach space of all absolutely $p$-summing complex sequences with index in $A_{0}$. Note that if the cardinal of $A$ is finite, lets say $n$, we have that $A_{0}$ can be identified with the set $\{-n,-n+1,\ldots,-2,-1,1,2,\ldots,n-1,n\}$ and if the cardinal of $A$ is infinite, $A_{0}$ can be identified with the set $\mathbb Z\setminus \{0\}$.
	Then, any element $z$ in $\ell_p(A_{0})$ can be written as 
	\[
		z=(z_{-n},\ldots,z_{-2},z_{-1},z_1,z_2,\ldots,z_n),
	\]
	if $n$ is finite, and 
	\[
		z=(\ldots,z_{-n},\ldots,z_{-2},z_{-1},z_1,z_2,\ldots,z_n,\ldots),
	\]
	if $n$ is infinite.
	Consider the following set of permutations of $A_{0}$ given by 
	\[
		\Lambda :=\{\sigma: A_{0}\mapsto A_{0}\text{ such that } \sigma \text{ is a permutation, and } \sigma(A)=A \text{ and }\sigma(-A)=-A\}.
	\]

Note that $\Lambda$ is a subgroup of permutations of the countable set $A_{0}$ and it can be considered to be a group action on $\ell_p (A_{0})$ in a natural way. Therefore, if $A$ is finite the results obtained in Section \ref{permutations} follow automatically. However, here we are interested in a specific family of $\Lambda$-invariant polynomials. Consider the following polynomials defined on $\ell_p(A_{0})$, $1 \leq p < \infty$, 
	\[
		T_k(z)=\sum_{i\in A}z_i^k-\sum_{i\in -A}z_i^k
	\] for $k\in\mathbb N$. It is clear by definition that $T_k (z) = T_k (\widehat{\sigma} (z))$ for all $\sigma \in \Lambda$; hence $T_k$ is a continuous $\Lambda$-invariant polynomial provided $k\geq p$. 
	
A polynomial $P$ on $\ell_p(A_{0})$ is said to be \emph{supersymmetric} if it can be represented as an algebraic combination of polynomials $(T_k)_{k=1}^\infty$. In other words, $P$ is a finite sum of finite products of polynomials in $(T_k)_{k=1}^\infty$ and constants.

	Supersymmetric polynomials were studied in \cite{Ser, Ste} for the case where $A$ is finite. In \cite{JZ} the authors consider the case where $A$ is infinite, and they study algebraic basis of the space of supersymmetric polynomials and the spectrum of algebras of supersymmetric analytic functions. Indeed, supersymmetric polynomials have applications in several areas of mathematics and physics such as the study of infinite generated Brauer groups and in the modelling of the behaviour of ideal gas in Statistical Mechanics, see \cite{JZ, Ser, Ste} and the references therein for the details. 
	Almost the same techniques to the ones used in the previous Proposition \ref{prop:permutations:ell_p} allow us to obtain the following result.
	
	\begin{prop}
	Let $A \subset \N$ be a set and set $X = \ell_p (A_{0})$ for $1 \leq p < \infty$. Suppose that $z$ is an element in $X$ such that either exits $j_0\in A$ such that $\vert x_{j_0}\vert >1$ and $\vert x_{j}\vert \leq 1$ for all $j\in -A$ or exits $j_0\in -A$ such that $\vert x_{j_0}\vert >1$ and $\vert x_{j}\vert \leq 1$ for all $j\in A$.
	Then $z$ can be separated from $\overline{B_X}$ by a supersymmetric polynomial $T$. 
	\end{prop}

	\begin{remark}
	Note that in general there are points outside the unit ball of $\ell_p(A_{0})$ that cannot be separated from $\overline{B_{\ell_p(A_{0})}}$ by a supersymmetric polynomial. Consider, for instance, the point $z$ in $\ell_p (A_{0})$ such that $z_j = z_{-j}$ for every $j \in A$. Then $z$ does not satisfy the assumptions the previous proposition. Moreover, $T_k (z) = 0$ for every $k \in \N$, which implies that $P(z) = 0$ for every supersymmetric polynomial $P$ on $\ell_p (A_{0})$. 
	\end{remark}


\subsection{The group of composition operators on $C(K)$.}
	Let us consider $C(K)$, the space of continuous functions on a compact Hausdorff space $K$, and the group $H \subset \Lin(C(K))$ of composition operators on $C(K)$ defined as 
	\[
	H= \{ \gamma : C(K) \rightarrow C(K) \text{ such that } \gamma(f)=f\circ \phi \,\, \text{for some homeomorphism} \,\, \phi : K \rightarrow K \}.
	\]
	This group was studied in \cite[Section 3]{AGPZ} and it is proved that, in the case of $K=[0,1]$, there exists a continuous non-multiplicative projection from $\mathcal{H}_b (C[0,1])$ onto the algebra of $H$-symmetric (bounded type) analytic functions on $C[0,1]$, endowed with the topology of uniform convergence on bounded subsets of $C[0,1]$. We present the following result which characterizes elements in $C[0,1] \setminus \overline{B_{C[0,1]}}$ that are separated from $\overline{B_{C[0,1]}}$ by a continuous $H$-invariant polynomial.

\begin{theorem}\label{separation:C[0,1]}
Let $g$ be an element outside $\overline{B_{C[0,1]}}$. Then $g$ can be separated from $\overline{B_{C[0,1]}}$ by an $H$-invariant continuous polynomial $P$ if and only if $\max \{|g(0)|, |g(1)| \} > 1$. 
\end{theorem}
	
	\begin{proof}
	Suppose that $\max \{|g(0)|, |g(1)| \} \leq 1$. As $|g(t_0)| > 1$ for some $t_0 \in (0,1)$, notice that the element $g$ is separated from $\overline{B_{C[0,1]}}$ by the evaluation functional $\delta_{t_0} \in C[0,1]^*$. However, no $H$-invariant holomorphic function on $C[0,1]$ separates $g$ and $\overline{B_{C[0,1]}}$. Indeed, \cite[Theorem 3.1]{AGPZ} shows that an $H$-invariant holomorphic function $F$ on $C[0,1]$ satisfies that 
	\[
	F(f) = F(t \mapsto f(0)(1-t) + f(1)t) \quad (f \in C[0,1]).
	\]
	Let us consider $\tilde{g} \in C[0,1]$ defined as $\tilde{g}(t) = g(0)(1-t) + g(1)t$ for every $t \in [0,1]$. Then $\|\tilde{g}\| \leq 1$, which implies that $\tilde{g} \in \overline{B_{C[0,1]}}$. As $F(g) = F(\tilde{g})$; hence $|F(g)| = |F(\tilde{g})| \leq \sup_{f \in \overline{B_{C[0,1]}}} |F(f)|$. It follows that an $H$-invariant holomorphic function $F$ on $C[0,1]$ cannot separate the element $g$ and the ball $\overline{B_{C[0,1]}}$. 
	
	Conversely, assume that $\max \{|g(0)|, |g(1)| \} > 1$, say $|g(0)|> r > 1$. For each $m \in \N$, consider the continuous polynomial $P_m : C[0,1] \rightarrow \C$ defined as 
	\[
	P_m(f) := \frac{1}{2} (f(0)^m + f(1)^m) \quad (f \in C[0,1]).
	\] 
	As 
	\[
	P_m (\gamma(f)) = P_m (t \mapsto (f\circ \phi) (t)) = \frac{1}{2} [f(\phi(0))^m + f(\phi(1))^m] = \frac{1}{2} (f(0)^m + f(1)^m) = P_m(f)
	\]
	for every $\gamma \in H$ (with associated homeomorphism $\phi$) and $f \in C[0,1]$, the continuous polynomial $P_m$ is an $H$-invariant continuous polynomial. Note that
	\[
	|P_m (f)| = \left| \frac{1}{2} (f(0)^m + f(1)^m) \right| \leq 1 
	\]
	for every $f \in \overline{B_{C[0,1]}}$ and $m \in \N$. However, if we denote $g(0)=|g(0)| e^{i \theta_0}$ and $g(1) = |g(1)| e^{i \theta_1}$ for some real numbers $\theta_0$ and $\theta_1$, then 
	\begin{equation}\label{estimate:Pm:composition_operators}
	|P_m (g)| = \left| \frac{1}{2} (g(0)^m + g(1)^m) \right| = 
	\frac{1}{2} \left| |g(0)|^m e^{i\theta_0 m} + |g(1)|^m e^{i\theta_1 m} \right|. 
	\end{equation}
	By Lemma \ref{lem_AFM}, there exists a sequence $(m_l)_{l=1}^{\infty} \subset \N$ such that $e^{i \theta_0 m_l} \xrightarrow[l \rightarrow \infty]{} 1$ and $e^{i \theta_1 m_l} \xrightarrow[l \rightarrow \infty]{} 1$. Pick $l_0 \in \N$ so that $|1 - e^{i\theta_0 m_l}| < \frac{1}{2}$ and $|1 - e^{i\theta_1 m_l}| < \frac{1}{2}$ for every $l \geq l_0$. 
	From \eqref{estimate:Pm:composition_operators}, we have for $l \geq l_0$ 
	\begin{align*}
	|P_{m_l}(g)| &\geq \frac{1}{2} \left(|g(0)|^{m_l} + |g(1)|^{m_l} - |g(0)|^{m_l} |1-e^{i\theta_0 m_l}| - |g(1)|^{m_l} |e^{i\theta_1 m_l} -1|	\right) \\
	&\geq \frac{1}{4} (|g(0)|^{m_l} + |g(1)|^{m_l}) \\
	&\geq \frac{1}{4} r^{m_l}. 
	\end{align*}
	It follows that $P_{m_l}$ with sufficiently large $l \in \N$ separates the element $g$ and the set $\overline{B_{C[0,1]}}$. 
	\end{proof}

In \cite{AGPZ} is also considered the space of $H$-invariant holomorphic functions on $C(K)$, where $K = [-1,1] \cup [0, i] \subset \C$ is the $T$-shape space, or $\{ e^{i\theta} \in \C : \theta \in [0,2\pi]\}$, i.e., the unit circle in $\C$, or $[0,1]^2 \subset \C$, i.e., the square in $\C$. It is proved that any $H$-invariant holomorphic function is constant when $K$ is the unit circle or square in $\C$. When $K$ is the $T$-shape space, an analytic function $F : C(K) \rightarrow \C$ is $H$-invariant if and only if there is a polynomial $\mathcal{F} \in \mathcal{H} (\C^4)$ symmetric with respect to the last three variables so that $F( f ) = \mathcal{F} (f(0), f(1), f(-1), f(i) )$ for every $f \in C(K)$. Following the idea of the proof of \cite[Example 3.6]{AGPZ}, we can observe that if $F : C(K) \rightarrow \C$ is an $H$-invariant polynomial, then 
\begin{equation}
\label{characterisation-of-T}
F (f) = F \left( f(0) ( 1 - \kappa_1 - \kappa_2 - \kappa_3) + f(1) \kappa_1 + f(-1) \kappa_2 + f(i) \kappa_3 \right) \quad (f \in C(K)),
\end{equation}
where $\kappa_1, \kappa_2$ and $\kappa_3$ are elements in $C(K)$ defined as 
\begin{align*}
\kappa_1 ( a+ ib) = a \chi_{[0,1]}(a+ib), \,\,\, \kappa_2 ( a+ ib) = -a \chi_{[-1,0]}(a+ib), \,\,\, \kappa_3 ( a+ ib) = b \quad ( a + ib \in K).
\end{align*}
The following result is an analogue of Theorem \ref{separation:C[0,1]} for the space $C(K)$ where $K$ the $T$-shape set. 

\begin{theorem}
	Let $K$ be the $T$-shape set and $g$ be an element outside $\overline{B_{C(K)}}$. Then $g$ can be separated from $\overline{B_{C(K)}}$ by an $H$-invariant continuous polynomial $P$ if and only if \[\max \{|g(0)|, |g(1)|, |g(-1)|, |g(i)| \} > 1.\] 
\end{theorem} 

\begin{proof}
	Suppose that $\max \{|g(0)|, |g(1)|, |g(-1)|, |g(i)| \} \leq 1$. Let $F$ be an $H$-invariant holomorphic function on $C(K)$. We know by \eqref{characterisation-of-T} that $F (f) = F(\tilde f)$ where 
\[
	\tilde{f}(w) = f(0) ( 1 - \kappa_1 (w) - \kappa_2 (w) - \kappa_3 (w)) + f(1) \kappa_1 (w) + f(-1) \kappa_2 (w) + f(i) \kappa_3 (w) 
	\]
	for every $w \in K$ and every $f\in C(K)$. Hence, if $f\in C(K)$ and $\max \{|g(0)|, |g(1)|, |g(-1)|, |g(i)| \} \leq 1$ we have that $\Vert \tilde g\Vert\leq 1$. But, $\vert F(g)\vert =\vert F(\tilde g)\vert \leq \sup_{f\in\overline{B_{C(K)}}}\vert F(f)\vert$. Therefore, no $H$-invariant holomorphic function on $C(K)$ separates the element $g$ and the ball $\overline{B_{C(K)}}$.
	
	Conversely, assume that $\max \{ |g(0)|, |g(1)|, |g(-1)|, |g(i)| \} > 1$. 
	\begin{enumerate}
		\item[(i)] If $|g(0)| > r > 1$, let us consider the continuous polynomial $P_m : C(K) \rightarrow \C$ defined as 
		\[
		P_m (f) = \frac{1}{4} \left( f(0)^m + f(1) + f(-1) + f(i) \right) \quad (f \in C(K)). 
		\] 
		By the shape of the set $K$, we see that $\phi (0) = 0$ and $\{ \phi(1), \phi(-1), \phi(i) \} = \{1, -1, i\}$ for every homeomorphism $\phi : K \rightarrow K$. Hence, $P_m$ is an $H$-invariant continuous polynomial. Note that 
		\[
		|P_m (f)| = \left| \frac{1}{4} \left( f(0)^m + f(1) + f(-1) + f(i) \right)	\right| \leq 1
		\]
		for every $f \in \overline{B_{C(K)}}$ and $m \in \N$. However,
		\begin{align*}
		|P_m (g)| = \left| \frac{1}{4} \left( g(0)^m + g(1) + g(-1) + g(i) \right)	\right| &\geq \frac{1}{4} ( |g(0)|^m - | g(1) + g(-1) + g(i)| ) \\
		&\geq \frac{1}{4} ( r^m - | g(1) + g(-1) + g(i)| ) \xrightarrow[m \rightarrow \infty]{} \infty;
		\end{align*} 
		hence, $P_m$ with sufficiently large $m$ can separate the element $g$ and the ball $\overline{B_{C(K)}}$. 
		\item[(ii)] If $\max \{|g(1)|, |g(-1)|, |g(i)| \} > r > 1$, then consider the continuous polynomial $R_m$ from $C(K)$ to $ \C$ defined as 
		\[
		R_m (f) = \frac{1}{4} \left( f(0) + f(1)^m + f(-1)^m + f(i)^m \right) \quad (f \in C(K)). 
		\] 
		Note that $R_m$ is an $H$-invariant continuous polynomial and $
		|R_m (f)| \leq 1$ for every $f \in \overline{B_{C(K)}}$ and $m \in \N$. Let $\theta_1, \theta_{-1}$ and $\theta_{0}$ be real numbers such that $g(1) = |g(1)| e^{i\theta_1}, g(-1)=|g(-1)| e^{i\theta_{-1}}$ and $g(i) = |g(i)| e^{i\theta_0}$. Using Lemma \ref{lem_AFM}, we can find a sequence $(m_l)_{l=1}^{\infty} \subset \N$ so that 
		$e^{i \theta_1 m_l} \xrightarrow[l \rightarrow \infty]{} 1$, $e^{i \theta_{-1} m_l} \xrightarrow[l \rightarrow \infty]{} 1$ and $e^{i \theta_0 m_l} \xrightarrow[l \rightarrow \infty]{} 1$. Choose $l_0 \in \N$ such that $|1 - e^{i\theta_1 m_l}| < \frac{1}{2}$, $|1 - e^{i\theta_{-1} m_l}| < \frac{1}{2}$ and $|1 - e^{i\theta_0 m_l}| < \frac{1}{2}$, then 
		\begin{align*}
		|R_{m_l} (g)| &= \left|\frac{1}{4} \left( g(0) + g(1)^{m_l} + g(-1)^{m_l} + g(i)^{m_l} \right) \right| \\
		&= \frac{1}{4} \left| g(0) + |g(1)|^{m_l} e^{i\theta_1 m_{l}} + |g(-1)|^{m_l} e^{i\theta_{-1} m_{l}} + |g(i)|^{m_l} e^{i\theta_0 m_{l}} \right| \\
		&\geq \frac{1}{4} \Big(|g(1)|^{m_l} + |g(-1)|^{m_l} + |g(i)|^{m_l} \\
		&\quad\qquad - |g(1)|^{m_l} |1-e^{i\theta_1 m_l}| - |g(-1)|^{m_l} |1-e^{i\theta_{-1} m_l}| - |g(i)|^{m_l} |1-e^{i\theta_0 m_l}| - |g(0)| \Big) \\ 
		&\geq \frac{1}{8} (|g(1)|^{m_l} + |g(-1)|^{m_l} + |g(i)|^{m_l}) - \frac{1}{4} |g(0)| \\
		&\geq \frac{1}{8} r^{m_l} - \frac{1}{4} |g(0)|
		\end{align*}
		for every $l \geq l_0$. This implies that the $H$-invariant continuous polynomial $R_{m_l}$ with sufficiently large $l \in \N$ separates the element $g$ and the set $\overline{B_{C(K)}}$. 
	\end{enumerate}
\end{proof}

	\subsection{The group of composition of measure-preserving maps on $L^p[0,1]$.}
		Let us consider the space $X = L^p [0,1]$ for $1 \leq p < \infty$. We study the group $M$ of $\Lin (L^p[0,1])$ defined as 
\[
M = \{ T : L^p [0,1] \rightarrow L^p [0,1] \text{ such that } Tx = x \circ \phi, \,\, \phi: [0,1] \rightarrow [0,1] \text{ a measure-preserving mapping}\},
\]
with the topology induced by $\Lin(L^{p}[0,1])$ was studied in \cite[Section 5]{AGPZ}.

Note that $x \in L^p [0,1] \rightarrow \int_0^1 x^k \in \C$ is a $k$-homogeneous $M$-invariant continuous polynomial for each $1 \leq k \leq p$. 
Given $N \in \N$, let $I_{j}^{(N)} \subset [0,1]$ be the interval $((j-1)2^{-N}, j 2^{-N})$ for $j = 1, \ldots, 2^N$. Note that the measure of this interval $I_j^{(N)}$ is $2^{-N}$ and these intervals form a partition of $[0,1]$. As in \cite{AGPZ}, let us denote by $S_N$ the space of $N$-level step functions defined as 
\[
S_N = \left\{ x : [0,1] \rightarrow \C : x(t) = \sum_{j=1}^{2^N} a_j \chi_j^{(N)} (t) \,\, \text{for some finite sequence} \,\, \{a_j\}_{j=1}^{2^N} \subset \C \right\}, 
\]
where $\chi_j^{(N)} (t) = \chi_{I_{j}^{(N)}} (t)$. Note that $\cup_{N=1}^{\infty} S_N$ is dense in $L^p [0,1]$. 
Let $\iota_N : \C^{2^N} \rightarrow S_N$ be the identification map defined by 
\[
z \mapsto x(t) = \sum_{j=1}^{2^N} z_j \chi_j^{(N)} (t), \quad z = (z_j)_{j=1}^{2^N} \in \C^{2^N}.
\]
It is observed in \cite[Corollary 5.4]{AGPZ} that if $P_k$ is a $k$-homogeneous $M$-invariant continuous polynomial with $k \leq p$, then there exists a continuous polynomial $Q : \C^k \rightarrow \C$ such that 
\[
P_k (x) = Q\left( \int_0^1 x, \ldots, \int_0^1 x^k \right) \quad (x \in L^p [0,1]).
\]

\begin{prop} 
For $1 \leq k \leq p$, if $z \in L^p [0,1]$ satisfies that $\max_{1 \leq j \leq k} \left| \int_0^1 z^j \right| > 1$, then $z$ can be separated from $\overline{B_{L^p [0,1]}}$ by an $M$-invariant continuous polynomial $P$. 
\end{prop}

\begin{proof}
Suppose that $\left| \int_0^1 z^j \right| > 1$ for some $1 \leq j \leq k$. 
Consider the continuous polynomial $P$ on $L^p [0,1]$ given by 
\[
P (x) = \int_0^1 x^j \quad (x \in L^p [0,1]).
\]
It is clear that $P$ is an $M$-invariant continuous polynomial. Note that 
\begin{align*}
|P(x)| \leq \int_0^1 |x|^j \leq \left(\int_0^1 (|x|^j)^{\frac{p}{j}} \right)^{\frac{j}{p}} \left(\int_0^1 1^{\left(1-\frac{j}{p}\right)^{-1}} \right)^{1-\frac{j}{p}} = \|x\|_p^j \leq 1
\end{align*} 
for every $x \in \overline{B_{L^p [0,1]}}$. 
However,
\[
|P(z)| = \left| \int_0^1 z^j \right| > 1 
\]
which implies that $P$ separates $z$ and $\overline{B_{L^p [0,1]}}$. 
\end{proof} 

Given $N \in \N$, if $x = \sum_{j=1}^{2^N} z_j \chi_j^{(N)} \in S_N$ and $\sigma \in \sym \{1,\ldots,2^N\}$, let us denote by $x \circ \sigma$ the element in $S_N$ given by 
\[
x \circ \sigma = \sum_{j=1}^{2^N} z_{\sigma(j)} \chi_j^{(N)}.
\]

We say that a subset $K$ of $L^p [0,1]$ is symmetric in the following sense:
\begin{equation}\label{symmetric_set_in_L^p[0,1]}
\text{if} \,\, x \in K \,\, \text{belongs to} \,\, S_N \,\, \text{for some} \,\, N \in \N\,\, \text{then} \,\, x \circ \sigma \in K \,\, \text{for every} \,\, \sigma \in \sym\{1,\ldots, 2^N\}.
\end{equation} 

Using similar techniques to the ones we used in the proof of Theorem \ref{complex_separation} and \cite[Lemma 5.1]{AGPZ}, we can obtain the following result.
\begin{prop}
Let $K$ be a symmetric set in $\cup_{N=1}^{\infty} S_N \subset L^p [0,1]$. If $z$ is an element in $L^p [0,1] \setminus K$ that can be separated from $K$ by a $k$-homogeneous continuous polynomial $Q$, then there exists a $M$-invariant continuous polynomial $P$ that separates $z$ and $K$. 
\end{prop} 

		\noindent \textbf{Acknowledgment:} This paper was partially written when the third author was visiting the University of Valencia and he would like to thank the hospitality that he received there.

	\end{document}